\documentclass[10pt]{amsart}
\usepackage{amssymb}
\usepackage{amsmath}
\usepackage{amsthm}
\usepackage{latexsym}
\usepackage{graphicx}
\usepackage{hyperref}

\newcommand{\Z}{{\mathbb Z}}
\newcommand{\R}{{\mathbb R}}
\newcommand{\C}{{\mathbb C}}
\newcommand{\D}{{\mathbb D}}
\newcommand{\T}{{\mathbb T}}

\newtheorem{lemma}{Lemma}[section]
\newtheorem{theorem}[lemma]{Theorem}

\newtheorem{remark}[lemma]{Remark}
\newtheorem{proposition}[lemma]{Proposition}
\newtheorem{corollary}[lemma]{Corollary}
\newtheorem{definition}[lemma]{Definition}

\newcommand{\nn}{\nonumber}
\newcommand{\be}{\begin{equation}}
\newcommand{\ee}{\end{equation}}
\newcommand{\ul}{\underline}
\newcommand{\ol}{\overline}
\newcommand{\ti}{\tilde}

\newcommand{\spr}[2]{\left\langle #1 , #2 \right\rangle}

\newcommand{\E}{\mathrm{e}}
\newcommand{\I}{\mathrm{i}}

\newcommand{\im}{\mathrm{Im}}
\newcommand{\re}{\mathrm{Re}}

\DeclareMathOperator{\dist}{dist}

\newcommand{\eps}{\varepsilon}

\numberwithin{equation}{section}

\begin{document}

\title[Small skew-shift]{An explicit skew-shift Schr\"odinger operator with positive Lyapunov exponent at small coupling}

\author[H.\ Kr\"uger]{Helge Kr\"uger}
\address{Mathematics 253-37, Caltech, Pasadena, CA 91125}
\email{\href{helge@caltech.edu}{helge@caltech.edu}}
\urladdr{\href{http://www.its.caltech.edu/~helge/}{http://www.its.caltech.edu/~helge/}}

\thanks{H.\ K.\ was supported by a fellowship of the Simons foundation.}

\date{\today}

\keywords{CMV matrices, spectrum, skew-shift, localization}

\begin{abstract}
 I give an example of a skew-shift Schr\"odinger operator with positive
 Lyapunov exponent in the spectrum for all positive coupling constant
 with constant frequency. This is the first example of this kind.

 The proof is based on CMV operators given by the skew-shift. Further
 results on these are derived.
\end{abstract}

\maketitle

%
%
%

\section{Introduction}

In the theory of one dimensional ergodic Schr\"odinger operators
$H_{g;\omega} = \Delta + g V_{\omega}$, there are two main
regimes defined by the Lyapunov exponent
\be
 L_{g} (E) = \lim_{N\to\infty} \mathbb{E}\left(\frac{1}{N} \log\left\|
  \prod_{n=N}^{1} \begin{pmatrix} E - g V_{\omega}(n) & -1 \\ 1 & 0 \end{pmatrix}
   \right\|\right)
\ee
either being positive on the spectrum or vanishing on it.
For $V_{\omega}(n) = f(T^n \omega)$, $\omega\in\Omega$,
and $T:\Omega\to\Omega$ a $\mu$-ergodic transformation,
there exists a set $\Sigma_{\mathrm{ac}; g} \subseteq \R$ such that
for $\mu$ almost-every $\omega$, the absolutely continuous spectrum
of $H_{g;\omega}$ is $\Sigma_{\mathrm{ac}; g}$. We then have that
$\Sigma_{\mathrm{ac}; g}$ is equal to the essential closure of
the set of energies $E \in \R$ such that the Lyapunov exponent
$L_{g}(E)$ vanishes. This and related topics are usually known
as {\em Kotani theory}, see \cite{dfest}.

The case of $V(n)$ being independent identically distributed random
variables has been understood for some time now: the Lyapunov
exponent is positive.

Much progress has been made in understanding the Lyapunov exponent
for quasi-periodic Schr\"odinger operators, that is with potential
given by $V_x(n) = f(x + n \alpha)$, where $f:\T \to \R$ is a real
analytic function, $\T=\R/\Z$, $x\in\T$, $\alpha$ 
irrational, $\mathbb{E}(h) = \int_{\T} h(x)dx$.
In particular, one has a transition from
vanishing Lyapunov exponent $L_g(E)$ on the spectrum for small $g > 0$
to positive Lyapunov exponent at large $g$.

Somewhat surprising from this, is that one expects that the Schr\"odinger
operator with potential
\be\label{eq:potomegasq}
 V(n) = \cos(2\pi \omega n^2)
\ee
should have positive Lyapunov exponent for $\omega$ irrational and all
$g > 0$, see \cite{cs} and \cite[Chapter~15]{bbook}. An adaptation of the methods
used to show that the Lyapunov exponent is positive for large
$g$ for quasi-periodic Schr\"odinger operators \cite{bbook,b2002,bgs}, yields that the Lyapunov
exponent is positive when $\omega$ is Diophantine, i.e.
\be\label{eq:diop}
 \kappa = \inf_{\Z\ni q\geq 1} q^2 \|q \alpha\| > 0
\ee
where $\|x\|=\dist(x,\Z)$ and the largeness condition on $g$
depends on $\kappa> 0$. The much harder and open problem is to show that the
Lyapunov exponent is positive for small $g > 0$.

\bigskip

Potentials of the form \eqref{eq:potomegasq} are of the general form
\be\label{eq:defpotskew}
 V_{g; \ul x}(n) = g\cdot f(T_{\omega}^n \ul x),
\ee
where $g > 0$ is a coupling constant, $\ul{x}\in\T^r$, $f:\T^r\to\R$,
and $T_{\omega}: \T^r \to \T^r$ is the skew-shift given by
\be
 (T_{\omega}\ul x)_{\ell} = \begin{cases} x_1 + \omega, & \ell =1;\\
  x_{\ell} + x_{\ell-1}, &2\leq\ell\leq r.\end{cases}
\ee
If $\omega$ obeys the Diophantine condition \eqref{eq:diop}
and $g > 0$ is sufficiently large depending on $\kappa$,
it follows from either \cite{b2002} or \cite{bgs} that the
Lyapunov exponent is positive and Anderson localization holds
in a suitable set of parameters for $r = 2$. Positive Lyapunov
exponent for general $r \geq 3$ is proven in \cite{kthesis}.
Furthermore, it is shown in \cite{kskew} that the spectrum
of $H_{g; \ul{x}} = \Delta + V_{g;\ul{x}}$
contains intervals for $r=2$ and $f$ only depending on $x_2$.
Although this picture is still incomplete,
we have a fair amount of understanding. For example the what
happens if $\omega$ is Liouville has not been investigated
yet.

In the case $g > 0$ small far less is known. For $r = 2$,
$f(x_1,x_2) = \cos(2\pi x_2)$, and $g > 0$ small enough, Bourgain has shown in \cite{b,b2,b3}
that for small enough $\omega$ obeying \eqref{eq:diop}, the Lyapunov
exponent is positive on a large set, which contains some spectrum.
Furthermore, I have shown in \cite{kmulti} that for $\lambda > 0$
small enough, $r\geq 1$ large enough, and a sampling function
only depending on the last coordinate $x_r$, the Lyapunov exponent
is positive on a set of large measure. A simpler but non-quantitative
proof of this result can also be found in \cite[Chapter~4]{kthesis}.

One of the objectives of this paper will be to prove

\begin{theorem}\label{thm:awesome}
 Let $r\geq 2$, $\omega$ Diophantine, i.e. obeying \eqref{eq:diop},
 $g > 0$, and
 \be\label{eq:deff}
  f(\ul{x}) = \cos(2\pi x_r) - \cos(2\pi (x_r + x_{r-1})).
 \ee
 Then there exists $\eps = \eps(\kappa, g) > 0$ such that
 the Lyapunov exponent $L_g(E)$ corresponding to \eqref{eq:defpotskew}
 satisfies
 \be
  L_g(E) \geq  \frac{1}{4} \log\left(1 + g^2\right)
 \ee
 for $|E| \leq \eps$. Furthermore, $0$ is contained
 in the spectrum of the operator.
\end{theorem}

The key difference between this theorem and the ones known
so far is that, one can fix $r\geq 2$, $\omega$ obeying \eqref{eq:diop},
and $f$ and obtain positive Lyapunov exponent in an energy
region containing some of the spectrum. Except for numerical computations
of the Lyapunov exponent this is the strongest evidence so
far, that we should believe in the conjecture of \cite{cs}.
It should be furthermore be pointed out that the proof of 
Theorem~\ref{thm:awesome} shows the correct asymptotic behavior
of the Lyapunov exponent, that is $L_g(E) \sim g^2$ for $g$
small enough (and $E$ in a small $g$ dependent range), and
$L_g(E) \sim \log\frac{1}{g}$ for $g$ large.

Of course, the choice of $f$ in Theorem~\ref{thm:awesome} should
seem odd. The best explanation is that this is what comes out
of the proof. Furthermore, it should be noted that the papers
\cite{b,b2,b3} require $r=2$ and $f(x_1,x_2) = g\cos(2\pi x_2)$
to maintain a close connection to the almost Mathieu operator.

One of my hopes is that building on Theorem~\ref{thm:awesome},
further developments in the theory of skew-shift Schr\"odinger
operators will arise. For example, it is an intriguing question
how to adapt \cite{kskew} to prove that the spectrum contains
an interval around $0$. Unfortunately, the initial estimates needed
in \cite{kskew} do not hold here.

Finally, let me remark that the potential $V(n) = \lambda \cos(2\pi
\omega n^{\rho})$ which can be thought of as interpolating
between quasi-periodic and skew-shift potentials for
$\rho \in (1,2)$ can be understood, see \cite{b,kpa}.

\bigskip

The main input into the proof is given by \cite{kcmv1}, where
a certain family of orthogonal polynomials on the unit circle
is analyzed by essentially algebraic tricks. The associated
Verblunsky coefficients are given by
\be
 \alpha_{\ul{x}; n} = \lambda \E^{2\pi\I (T^n \ul{x})_r},
\ee
where $\lambda\in\mathbb{D}=\{z:\quad |z|<1\}$.
In particular, one obtains a CMV operator $\mathcal{E}_{\ul{x}}$
with Lyapunov exponent given by $- \frac{1}{2} \log(1 - |\lambda|^2)$.
The main realization is now that the problem for $\mathcal{E}_{\ul{x}}$
is equivalent to a tridiagonal operator, which has constant
off-diagonal terms when $z = -1$ and thus is a Schr\"odinger operator.
The result then essentially follows by taking the real part of 
the CMV operator, see Corollary~\ref{cor:subawesome}.

The results of \cite{kcmv1} are by themselve not strong enough yet
to imply Theorem~\ref{thm:awesome}. One further goal of this paper
is to improve this results, in particular to show that exponential
decay of the Green's function holds with super polynomially small
probability and thus one has as good results as in the Schr\"odinger
case.

\bigskip

This brings me to the second motivation for writing this paper. It
came as a surprise that the microscopic eigenvalue statistics
in the case $r = 2$ is very regular as shown in \cite{kcmv1}. I have
recently shown in \cite{kconev} that this is not the case for
Schr\"odinger operators. The results of this paper allow one to extent
this to CMV operators as discussed in \cite{kcmv1} when $r \geq 2$
and in particular one sees that the microscopic distribution of the
eigenvalues in much closer to the one of the Anderson model
than the one in the case $r = 2$.

\bigskip

The content of the rest of the paper can be described as follows.
In Section~\ref{sec:cmvop}, I discuss the background on CMV operators
necessary for the proof of Theorem~\ref{thm:awesome}. In particular,
in Corollary~\ref{cor:subawesome} a weak version of Theorem~\ref{thm:awesome}
is derived. Furthermore, Theorem~\ref{thm:mainA} and Theorem~\ref{thm:main1}
present results on CMV operators, which are of independent interest.

Section~\ref{sec:pfawesome} contains the proof of Theorem~\ref{thm:awesome}
and further discussions of properties of this Schr\"odinger operator.
Section~\ref{sec:semialg} proves facts about the semi-algebraic structure
of certain sets related to Green's function estimates for CMV operators.
Section~\ref{sec:proof} contains the proof of Theorem~\ref{thm:mainA},
which is the biggest technical progress given in this paper.
Finally, Appendix~\ref{sec:returntimes} recalls some facts about the
return times of the skew-shift to a semi-algebraic set.

%
%
%

\section{CMV operators}
\label{sec:cmvop}

In this section, we introduce large parts of the notation necessary
to prove Theorem~\ref{thm:awesome}, in particular CMV operators. As
much of the notation related to CMV operators is similar to the one
of Schr\"odinger operators, I will defer the proof of Theorem~\ref{thm:awesome}
to the next section. 

Let us now begin by introducing the necessary notation. Given a
bi-infinite sequence of Verblunsky coefficients $\alpha_n \in \mathbb{D}
= \{z:\quad |z|<1\}$, we define the matrices
\be
 \Theta_n = \begin{pmatrix} \ol{\alpha_n} & \rho_n \\ \rho_n 
   & - \alpha_n \end{pmatrix},\quad \rho_n = \sqrt{1-|\alpha_n|^2}
\ee
thought of as acting on $\ell^2(\{n,n+1\})$. Then, we define
the operators
\be
 \mathcal{L} = \bigoplus_{n\text{ even}} \Theta_n,\quad
  \mathcal{M} = \bigoplus_{n\text{ odd}} \Theta_n
\ee
and the extended CMV operator $\mathcal{E} = \mathcal{L} \cdot \mathcal{M}$.
One can easily check that $\mathcal{E} :\ell^2(\Z)\to\ell^2(\Z)$
is an unitary operator. CMV operators are usually discussed in
the context of orthogonal polynomials on the unit circle,
when it is more natural to consider half line objects. For our
purposes the whole line operator is more natual.
The orthogonal polynomial aspects of the theory have been worked out extensively,
see for example the books \cite{opuc1, opuc2} and
the extensive references therein.

For $\beta,\ti{\beta}\in \partial\D=\{z:\quad |z|=1\}$ and $a < b \in \Z$,
the restriction $\mathcal{L}^{[a,b]}_{\beta,\ti{\beta}}$
is defined by setting $\alpha_{a-1} = \beta$, $\alpha_{b} = \ti{\beta}$
and restricting the resulting CMV operator to $\ell^2([a,b])$.
$\mathcal{M}^{[a,b]}_{\beta,\ti{\beta}}$ is defined
in a similar way. Finally, one defines
\be
 \mathcal{E}^{[a,b]}_{\beta,\ti{\beta}} = 
  \mathcal{L}^{[a,b]}_{\beta,\ti{\beta}} \cdot
   \mathcal{M}^{[a,b]}_{\beta,\ti{\beta}}.
\ee
All these operators are unitary. Furthermore, it should be 
noted that $\beta$ and $\ti{\beta}$ take the role of
boundary conditions. To understand the content of the next
lemma, it is important to observe that $(z -\mathcal{E}) \psi = 0$
is equivalent to $(z \mathcal{L}^{\ast} - \mathcal{M}) \psi = 0$
as $\mathcal{L}$ is unitary.

\begin{lemma}\label{lem:tridiag}
 The matrix $A = z (\mathcal{L}^{[a,b]}_{\beta,\ti{\beta}})^{*}
  - \mathcal{M}^{[a,b]}_{\beta,\ti{\beta}}$ is tridiagonal.
 Write $A = \{A_{i,j}\}_{a\leq i,j \leq b}$.
 Then we have that
 \be
  A_{j,j} = \begin{cases}
   z \alpha_j + \alpha_{j-1},&j\text{ even};\\
  - z \ol{\alpha_{j-1}}
   - \ol{\alpha_{j}},&j\text{ odd},\end{cases}\quad
  A_{j+1, j} = A_{j,j+1} = \ti{\rho}_j = 
  \begin{cases} z \rho_j,&j\text{ even};\\
   - \rho_{j},&j\text{ odd}.\end{cases}
 \ee
\end{lemma}

By Lemma~\ref{lem:tridiag}, we have with $z = - 1$
that the matrix of the operator $B = \re(z (\mathcal{L}^{[a,b]}_{\beta,\ti{\beta}})^{*}
- \mathcal{M}^{[a,b]}_{\beta,\ti{\beta}})$ are given by
\be
 B_{j,j} = \re(\alpha_{j-1} - \alpha_{j}),\quad B_{j,j+1} = B_{j+1,j} = -\rho_j.
\ee
Let us now discuss the Verblunsky coefficients $\alpha_{\ul{x};n} = \lambda\exp(2\pi\I 
(T_{\omega}^n\ul{x})_r)$, where $T_{\omega}:\T^r\to\T^r$ is the skew-shift
with frequency $\omega$. Then, we have that $\rho_{j} = \sqrt{1 - |\lambda|^2}$ 
and that the diagonal elements are given by
\be
 B_{j,j} = V_{\ul{x}}(j) = \re\left(\lambda \left(
 \exp(2\pi\I (T_{\omega}^{j-1}\ul{x})_r)-
  \exp(2\pi\I  (T_{\omega}^j\ul{x})_r)\right)\right).
\ee
This clearly has the form of a skew-shift potential. 

\begin{theorem}
 Let $r \geq 2$ and $\lambda\in\D\setminus\{0\}$
 We have that $\sigma(\mathcal{E}_{\ul{x}})$ and for
 $z\in\partial\D$ that
 \be
  \ell_{\lambda}(z) = - \frac{1}{2} \log(1 - |\lambda|^2).
 \ee
\end{theorem}

\begin{proof}
 This is in \cite{kcmv1}.
\end{proof}

Here $\ell_{\lambda}$ denotes the Lyapunov exponent for
CMV operators given by
\be
 \ell_{\lambda}(z) = \lim_{N\to\infty} \frac{1}{N}\mathbb{E}
  \left(\log\left\|\prod_{n=1}^{N} \frac{1}{\rho_{\ul{x};n}} \begin{pmatrix}
   z & - \ol{\alpha_{\ul{x};n}} \\ - z \alpha_{\ul{x};n} & 1 \end{pmatrix}\right\| \right),
\ee
where $\mathbb{E}(h) = \int_{\T^r} h(\ul{x}) d\ul{x}$.
This theorem combined with the considerations preceding
it imply the following weak version of Theorem~\ref{thm:awesome}.

\begin{corollary}\label{cor:subawesome}
 Let $H_{g; \ul{x}}$ be the Schr\"odinger operator with
 potential as above. Then $0$ is in the spectrum and
 the Lyapunov exponent is given by 
 \be
  L_{g}(0) = \frac{1}{2} \log\left(1 + g^2\right).
 \ee
\end{corollary}

\begin{proof}
 We first note that
 \[
 - \frac{1}{\sqrt{1-|\lambda|^2}} (-\mathcal{L}^{*}_{\ul{x}} - \mathcal{M}_{\ul{x}} )
  = \Delta + V,
 \]
 where $V$ is given by \eqref{eq:defpotskew} and
 \[
  f(\ul{x}) = \re\left(\E^{2\pi\I (T_{\omega}^{-1}\ul{x})_r}
   - \E^{2\pi\I (\ul{x})_r}\right) = \cos(2\pi (T_{\omega}^{-1} \ul{x})_r) 
    - \cos(2\pi x_r),\quad g=\frac{\lambda}{\sqrt{1-|\lambda|^2}}.
 \]
 That $0\in\sigma(H_{g;\ul{x}})$ follows from spectral calculus.
 To see the claim about the Lyapunov exponent, observe that 
 for almost every $\ul{x}\in\T^r$ there exists $\psi$ such that
 \[
  (-\mathcal{L}^{*}_{\ul{x}} - \mathcal{M}_{\ul{x}} )\psi = 0
 \]
 and $\frac{1}{2n} \log(|\psi(n)|^2 +|\psi(n+1)|^2) \to \frac{1}{2} \log(1-|\lambda|^2)$.
 As $\psi$ also solves $H_{g;\ul{x}} \psi = 0$, the claim follows.
\end{proof}

The main problem in extending this to prove Theorem~\ref{thm:awesome}
is to achieve stability in the energy, we will do this by instead
proving a finite scale claim about the Green's function.
Let $z\in\C$, $\beta,\ti{\beta}\in\partial{\D}$,
$a\leq k,\ell\leq b$, then the Green's function is 
defined by
\be\label{eq:defgreen}
 G^{[a,b]}_{\beta,\ti{\beta}}(z; k,\ell)=
  \spr{\delta_k}{\left(z (\mathcal{L}^{[a,b]}_{\beta,\ti{\beta}})^{*}
   - \mathcal{M}^{[a,b]}_{\beta,\ti{\beta}}\right)^{-1} \delta_\ell}.
\ee
As discussed in \cite{kcmv1}, the advantage over
considering the matrix elements of $(z - \mathcal{E}^{[a,b]}_{\beta,\ti{\beta}})^{-1}$
is that an application of Cramer's rule yields
much simpler terms. By Lemma~3.9 in \cite{kcmv1}, we have
\begin{align}
\nn |\psi(n)| \leq &2 |G^{[a,b]}_{\beta,\ti{\beta}}(z; n, a)| \sup_{m\in\{a-1,a\}} |\psi(m)|\\
\label{eq:green2sol}
& + 2 |G^{[a,b]}_{\beta,\ti{\beta}}(z; n, b)| \sup_{m\in\{b,b+1\}} |\psi(m)|
\end{align}
if $\psi$ solves $\mathcal{E} \psi = z \psi$, $a < n < b$,
and $\beta,\ti\beta\in\partial\D$.

We describe the decay properties of the Green's function
using the following definition.

\begin{definition}\label{def:suit}
 Let $[-N,N] \subseteq \Z$ be an interval, $z,\beta,\ti{\beta}\in\partial\D$,
 $\Gamma > 0$, $\gamma > 0$, $p\geq 0$. Then
 $[a,b]$ is called $(\gamma, \Gamma, p)$-suitable for
 $\mathcal{E}_{\beta,\ti{\beta}} - z$, if
 \begin{enumerate}
  \item $\|(\mathcal{E}^{[-N,N]}_{\beta,\ti{\beta}} - z)^{-1}\| \leq \frac{1}{2^{p}}\E^{\Gamma}$.
  \item For $k, \ell \in [-N,N]$ with $|k - \ell| \geq \frac{N}{2}$,
   we have
   \be
    |G_{\beta,\ti{\beta}}^{[a,b]}(z; k, \ell)| \leq \frac{1}{2^{p+1}}
    \E^{-\gamma |k-\ell|}.
   \ee
  \end{enumerate} 
\end{definition}

We are now ready for

\begin{theorem}\label{thm:mainA}
 Let $\lambda\in\D\setminus\{0\}$, $z,\beta,\ti{\beta}\in\partial\D$,
 and assume $\omega$ obeys
 \eqref{eq:diop}. Consider the CMV operator with
 Verblunsky coefficients given by 
 \be
  \alpha_{\ul{x};n}=\lambda\exp(2\pi\I (T^n_{\omega}\ul{x})_r). 
 \ee
 There exist constants $\sigma \in (0,1)$,
 $\tau \in (0,1)$, and $\gamma > 0$
 such that for $N$ sufficiently large
 \be
  |\{\ul{x}\in\T^r:\quad [-N,N]\text{ is $(\gamma,N^{\tau},p)$-unsuitable for }
   \mathcal{E}_{\ul{x};\beta,\ti{\beta}} - z\}| \leq \E^{-N^{\sigma}}.
 \ee
\end{theorem}

Once, this theorem is established proving Theorem~\ref{thm:awesome}
can be done by standard methods as discussed in \cite{kthesis}.
We will discuss this in detail in the next section. 
The proof of Theorem~\ref{thm:mainA} is given in Sections~\ref{sec:semialg}
and \ref{sec:proof}.

I will now explain some more consequences for CMV operators
of Theorem~\ref{thm:mainA}. The following theorem is essentially just an
application of appropriate perturbation theory. For a simplified
statement, we define
\be
 f_0(\ul{x}) = \lambda \exp(2\pi\I (\ul{x})_r)
\ee
and the $w > 0$ neighborhood of $\R^r$ by
\be
 \mathcal{A}_{w} = \{\ul{z}\in\C^r:\quad |\im(z_j)|\leq w\}.
\ee
A function $f:\mathcal{A}_w \to \C$ is one periodic
if $f(\ul{z} + \ul{n}) = f(\ul{z})$ for all $\ul{n}\in\Z^r$ and $\ul{z}\in\mathcal{A}_w$.

\begin{theorem}\label{thm:main1}
 Let $r \geq 2$, $w > 0$, $\kappa > 0$, $\omega$ such that \eqref{eq:diop}
 holds. Then there exist $\gamma =\gamma(\kappa, r,w, \lambda) > 0$, and 
 $\eps_0 = \eps_0(\kappa,r,w,\lambda) >0$ such that if for
 $f:\mathcal{A}_w \to \C$ analytic and one periodic with
 \be
  \|f - f_0\|_{w} = \sup_{|\im(x_j)| \leq w}
   |f(\ul{x}) - f_0(\ul{x})| < \eps_0,
 \ee
 we have that there exists a set $\mathcal{B}$ such that
 the conclusions of Theorem~\ref{thm:mainA} hold for the
 Verblunsky coefficients
 \be
  \alpha_{\ul{x};n} = f(T_{\omega}^{n} \ul{x}).
 \ee
\end{theorem}

In order to give a proof of this result, we first need

\begin{lemma}\label{lem:perturbsuit}
 Let $[a,b]$ be $(\gamma,\Gamma,p)$-suitable for $\mathcal{E}_{\beta,\ti{\beta}} - z$.
 Then if
 \be
  \sup_{n\in [a,b]} |\hat{\alpha}_n - \alpha_n|, |\hat{z} - z| \leq \E^{- (2 \Gamma + \gamma |b-a|)}
 \ee
 we have
 \be
  [a,b]\text{ is $(\gamma,\Gamma,p-1)$-suitable for }\widehat{\mathcal{E}}_{\beta,\ti{\beta}} - \hat{z}.
 \ee
\end{lemma}

\begin{proof}
 This is just an application of the resolvent equation and some estimates.
\end{proof}

\begin{proof}[Proof of Theorem~\ref{thm:main1}]
 The previous lemma implies that the conclusions for fixed $N$ of Theorem~\ref{thm:mainA}
 are stable under small perturbations of $f$ in the $\|.\|_w$ topology.
 Then using the results of Sections~\ref{sec:semialg} and \ref{sec:proof},
 it is possible to adopt the results of \cite{kthesis} to show that the
 conclusion of Theorem~\ref{thm:mainA} for $1 \leq N \leq N_0$ for some
 large enough $N_0$ implies them for all $N \geq 1$.
\end{proof} 

Furthermore, using the results of \cite[Chapter~15]{bbook},
it is possible to show Anderson and dynamical localization
for these CMV operators as long as $r = 2$.

%
%
%

\section{Proof of Theorem~\ref{thm:awesome}}
\label{sec:pfawesome}

As $- \log(1 - t) \geq t$ for $t > 0$, if we establish that the
Lyapunov exponent is continuous, we obtain that
\be
 L_{g}(E) \geq \frac{g^2}{4}
\ee
for $g > 0$ small and $E$ small by Corollary~\ref{cor:subawesome}.
Of course, establishing continuity of the Lyapunov exponent will be a
non-trivial task. Our proof will essentially follow the ideas of
\cite{kthesis}.

For $H=\Delta +V:\ell^2(\Z)\to\ell^2(\Z)$ and $[a,b]\subseteq\Z$
an interval, we define $H^{[a,b]}$ to be the restriction of
$H$ to $\ell^2([a,b])$. As $H$ is self-adjoint also $H^{[a,b]}$
is. For $z\in\C$, $k,\ell\in [a,b]$, we define the {\em Green's
function} by
\be
 G^{[a,b]}(z;k,\ell) = \spr{\delta_k}{(H^{[a,b]}-z)^{-1} \delta_{\ell}}.
\ee
As in the case of CMV operators, we could make all the definitions
with boundary conditions. However, I have decided not to do so,
as it isn't necessary to have self-adjoint operators, and to have
different notations for CMV and Schr\"odinger operators.
It is clear how to adapt Definition~\ref{def:suit} into
this new context. Furthermore, Theorem~\ref{thm:mainA} can
be seen to imply that

\begin{theorem}
 Let $\lambda > 0$ and $\omega$ obey \eqref{eq:diop}. Then
 there exist constants $\sigma,\tau \in (0,1)$, $\gamma > 0$
 such that for the Schr\"odinger operator with potential
 given by \eqref{eq:defpotskew} and \eqref{eq:deff}, we have
 \be
  |\{\ul{x}\in\T^r:\quad [-N,N]\text{ is not $(\gamma,N^{\tau}, 3)$-suitable for }
   H_{g;\ul{x}}\}| \leq \E^{-N^{\sigma}}
 \ee
 for $N \geq 1$ sufficiently large.
\end{theorem}

\begin{proof}
 By construction, we have as in Corollary~\ref{cor:subawesome} that
 \[
   H_{g;\ul{x}} = \frac{1}{\sqrt{1-\lambda^2}} \re\left(- \mathcal{L}_{\ul{x}} - \mathcal{M}_{\ul{x}}\right).
 \]
 So the result follows.
\end{proof}

Of course, this theorem only holds for $E = 0$. Let $N$ be
such that the conclusion holds. Then we see from
an easy perturbation argument as in the proof of Theorem~\ref{thm:main1}
that we have for $|E| \leq \E^{-3 N}$
that
\be\label{eq:ldeschroe}
 |\{\ul{x}\in\T^r:\quad [-N,N]\text{ is not $(\gamma,N^{\tau}, 2)$-suitable for }
  H_{g;\ul{x}} - E\}| \leq \E^{-N^{\sigma}}.
\ee
From this, we can conclude

\begin{theorem}
 There exists $\eps > 0$ such that for $|E|\leq \eps$, we have that
 \eqref{eq:ldeschroe} holds for $N \geq 1$ large enough.
\end{theorem}

\begin{proof}
 This follows by an iterative application of Theorem~7.4
 in \cite{kthesis}. Alternatively, one can also adapt the methods
 of Chapter~15 in \cite{bbook} to prove this result.
\end{proof}

This theorem implies by itself that the Lyapunov exponent is positive.
Unfortunately, it does not imply the correct size of the Lyapunov exponent,
as $\gamma > 0$ is much too small. This is the reason, we conclude the claim
of positivity of the Lyapunov exponent from continuity.

The integrated density of states is defined by
\be
 k(E) = \lim_{N\to \infty} \frac{1}{N} 
  \int_{\T^r} \#\{\text{eigenvalues of $H^{[1,N]}_{g;\ul{x}}$ } \leq E\}d\ul{x}.
\ee
Corollary~7.6. in \cite{kthesis} implies that there exists some $c > 0$ such that
\be
 k(E +\delta) - k(E -\delta) \leq \exp\left(- \log\left(\frac{1}{\delta}\right)^c\right)
\ee
for $|E| \leq \eps$ and $\delta > 0$ small enough. As the Lyapunov
exponent satisfies
\be
 L(E) = \int \log|E - t| dk(t),
\ee
we have that for $|E| \leq \eps$ the Lyapunov exponent has
similar continuity properties, and thus is continuous,
which implies our main claim. The relation between continuity
of the Lyapunov exponent and the integrated density of states
is discussed in Section~10 of \cite{gs01}.

\bigskip

Finally, let me remark that in case $r = 2$, the results of Chapter~15
in \cite{bbook} imply that Anderson localization holds. We take this to
mean that the spectrum of $H_{g;\ul{x}}$ is pure point in $[-\eps, \eps]$
and that the corresponding eigenfunctions $\psi_{\ell}$ satisfy
\be
 |\psi_{\ell}(n)| \leq C_{\ell} \E^{-\gamma |n|}
\ee
for some $C_{\ell} > 0$ and a uniform constant $\gamma > 0$.

%
%
%

\section{Semialgebraic Structure of Suitability}
\label{sec:semialg}

In this section, we investigate the geometric structure
of the set of suitabiliy for CMV operators.
We will classify how complicated a set is by its semi-algebraic
description. That is to say a set is simple if it can be described by
few polynomial equations involving polynomials of low degree.
Some background on these questions can be found in Chapter~9 of
\cite{bbook}. These methods were first introduced in \cite{bg}.
For general background see \cite{bcr}.

A set $\mathcal{S} \subseteq \T^K$ is a semi-algebraic set if 
there exist polynomials $P_1, \dots, P_b$, $Q_1, \dots, Q_B \in \R[x_1, \dots, x_K]$
such that
\be
 \mathcal{S} = \bigcup_{j=1}^{b} \{\ul{x}: \quad P_j(\ul{x}) \leq 0\} \cup
  \bigcup_{j=1}^{B} \{\ul{x}:\quad Q_j(\ul{x}) = 0 \}.
\ee
The degree of the semi-algebraic set $\mathcal{S}$ is
given by
\be
 \deg(\mathcal{S}) \leq (b + B) \cdot \max(\deg(P_j), \deg(Q_j)),
\ee
where the degree is the minimal value of the left-hand side.

The goal of this section is to show

\begin{theorem}\label{thm:semialgdesc}
 Let $\Gamma > 0$, $\gamma_{\infty} > 0$, $N\geq 1$ large enough,
 $|a|,|b|\leq N$, $\gamma_{k,\ell} \in [0,\gamma_{\infty} \cdot N]$ for
 $k,\ell \in [a,b]$, and $\beta,\ti{\beta}\in\partial\D$.
 Let $\mathcal{U}_p\subseteq\T^r$ be a
 set of $\ul{x}$ such that 
 \be
  \|(z(\mathcal{L}^{[a,b]}_{\ul{x}; \beta,\ti{\beta}})^{\ast}
    -\mathcal{M}^{[a,b]}_{\ul{x}; \beta,\ti{\beta}})^{-1} \| \leq \frac{1}{2^{p+1}} \E^{\Gamma}
 \ee
 and for $k,\ell\in [a,b]$ with $\gamma_{k,\ell} > 0$
 \be
  |G_{\ul{x};\beta,\ti{\beta}}^{[a,b]}(z;k,\ell)|
   \leq \frac{1}{2^{p}} \E^{-\gamma_{k,\ell}}.
 \ee

 Then there exists a semi-algebraic set $U$ of degree
 $N^{B}$ such that
 \be
  \mathcal{U}_{p-1} \subseteq U \subseteq \mathcal{U}_{p+1},
 \ee 
 where $B \geq 1$ is some universal constant.
\end{theorem}

How to apply this theorem can be seen in combination with Theorem~\ref{thm:gromov},
once one knows a measure estimate on $\mathcal{U}_{p+1}$. We now
begin the proof of Theorem~\ref{thm:semialgdesc}.

\bigskip
 
We recall that for us $f:\R^{r} \to \C$ is a real-analytic function,
that is there exists an analytic extension of $f$ to a neighborhood of
the form
\be
 \mathcal{A}_{w} = \{\ul{z}\in\C^{r}:\quad |\im(z_j)| \leq w,\ j=1,\dots, r\},
\ee
where $w > 0$.
Such an $f$ is one periodic if
\be
 f(\ul{z} +\ul{n}) = f(\ul{z})
\ee
for all $\ul{z}\in\mathcal{A}_{w}$ and $\ul{n}\in\Z^r$.
We define the following norm on these functions
\be
 \|f\|_{w} = \sup_{\ul z \in\mathcal{A}_{w}} |f(z)|.
\ee
We furthermore recall that $g:\R^r \to \R$ is a {\em trigonometric
polynomial} of degree $d$ if it can be written in the form
\be
 g(\ul{x}) = \sum_{\ul{k}\in\Z^r, |\ul{k}|_{\infty} \leq d} \hat{g}(\ul{k}) e(\ul{k} \cdot\ul{x})
\ee
where $|\ul{k}|_{\infty} = \sup_{1\leq j\leq r} |k_j|$,
$e(x) =\E^{2\pi\I x}$, and $\ul{k}\cdot\ul{x}=\sum_j k_jx_j$.
One easily sees that a trigonometric polynomial
is a one-periodic function $\mathcal{A}_{w} \to \C$ for
any $w > 0$.

Define
\be\begin{split}
 \widetilde{T}: \R^r &\to \R^r, \\
 (\widetilde{T}\ul{x})_\ell &=\begin{cases} 
  x_1 + \omega, &\ell =1;\\
  x_{\ell} + x_{\ell-1}, & 2 \leq\ell\leq r.\end{cases}
\end{split}\ee
Then the skew-shift is just $T\ul{x} = \widetilde{T}\ul{x} \pmod{1}$.
In particular, we have that
\be
 \alpha_{\ul{x}; n} = f(\widetilde{T}^n \ul{x}),\quad
  \rho_{\ul{x}; n} = \sqrt{1 - |f(\widetilde{T}^n \ul{x})|^2}
\ee
which is defined as long as $|f(\widetilde{T}^n \ul{x})| \leq 1$.

\begin{lemma}\label{lem:approxbytrigpoly}
 Let $n \in \Z$. There exist trigonometric polynomials $p_1$ and $p_2$ 
 of degree $n^{r-1} D$ such that
 \be
  |\alpha_{\ul{x}; n} - p_1(\ul{x})| \leq C \E^{- \frac{\rho}{2} D},\quad
  |\rho_{\ul{x}; n} - p_2(\ul{x})| \leq C \E^{- \frac{\rho}{2} \sqrt{D}}
 \ee
 for $\ul{x}\in\R^r$ and some universal constant $C > 0$.
\end{lemma}

\begin{proof}[Proof of Lemma~\ref{lem:approxbytrigpoly} for $\alpha_{\ul{x};n}$]
 One can easily check that the Fourier coefficients of $f$
 satisfy $|\hat{f}(\ul{k})|\leq \|f\|_{\rho} \E^{-\rho |\ul{k}|_{\infty}}$.
 Introduce
 \[
  \ti{p}_1(\ul{x}) = \sum_{|\ul{k}|\leq D} \hat{f}(\ul{k}) e(\ul{k} \cdot \ul{x})
 \]
 and $p_1(\ul{x}) = p_1(\widetilde{T}^n \ul{x})$. As the components
 $\widetilde{T}^n \ul{x}$ are polynomials of degree at most $n^{r-1}$,
 the claim follows.
\end{proof}

Even when $\alpha_{\ul{x}; n}$ is a polynomial in $\ul{x}$, it is
not clear that $\rho_{\ul{x}; n} = \sqrt{1 - |\alpha_{\ul{x}; n}|^2}$
is. We first note the Taylor series expansion
\be
 \sqrt{1 - x} = \sum_{n=0}^{\infty} \frac{(2n)!}{(1 - 2n) (n!)^2 4^n} x^n,
\ee
which converges for $|x| \leq 1$. As $\left|\frac{(2n)!}{(1 - 2n) (n!)^2 4^n}\right| \leq 1$,
we obtain for $r\in (0,1)$ the error estimate
\be\label{eq:taylorsqrt}
 \left|\sqrt{1 - x} - \sum_{n=0}^{N} \frac{(2n)!}{(1 - 2n) (n!)^2 4^n} x^n \right|
  \leq \frac{r^N}{1 - r}
\ee
for $|x| \leq r$ and $N \geq 1$.

\begin{proof}[Proof of Lemma~\ref{lem:approxbytrigpoly} for $\rho_{\ul{x};n}$]
 There is some $r \in (0,1)$ such that $|f(\ul{x})| \leq r$
 for all $\ul{x}\in\T^r$. By Lemma~\ref{lem:approxbytrigpoly}
 with $\lfloor D^{\frac{1}{2}} \rfloor$ in place of $D$, we thus obtain
 a trigonometric polynomial $p_1(\ul{x})$ such that
 \[
  |p_1(\ul{x})| \leq \frac{1 + r}{2} \in (0,1)
 \]
 for $\ul{x}\in\T^r$.
 Defining
 \[
  p_2(\ul{x}) = \sum_{n=0}^{\lfloor D^{\frac{1}{2}} \rfloor} 
   \frac{(2n)!}{(1 - 2n) (n!)^2 4^n} (p_1(\ul{x}))^n
 \]
 we obtain a trigonometric polynomial of degree $\leq D$
 that satisfies
 \[
  |p_2(\ul{x}) - \rho_{\ul{x}; n}| \lesssim \E^{- c \sqrt{D}}
 \]
 for some $c > 0$ by \eqref{eq:taylorsqrt}.
\end{proof}

Our main conclusion from Lemma~\ref{lem:approxbytrigpoly} is

\begin{corollary}\label{cor:approxzLM}
 There exists a constant $B > 0$ such that for $N \geq 1$ large enough,
 $\beta,\ti{\beta}\in\partial{D}$,
 and $|a|, |b| \leq N$, there exists a matrix $A(\ul{x})$ whose
 entries are polynomials of degree less than $N^{B}$ such that
 \be
  \|(z(\mathcal{L}^{[a,b]}_{\ul{x}; \beta,\ti{\beta}})^{\ast}
    -\mathcal{M}^{[a,b]}_{\ul{x}; \beta,\ti{\beta}}) - A(\ul{x})\|
    \leq \E^{-N^2}.
 \ee
\end{corollary}

Given a matrix $A \in \C^{N \times N}$, we define its Hilbert--Schmidt
norm by
\be
 \|A\|_{\mathrm{HS}} = \left(\sum_{n=1}^{N} \sum_{m=1}^{N} |A_{n,m}|^2\right)^{\frac{1}{2}}.
\ee
It is well known that we have $\|A\| \leq \|A\|_{\mathrm{HS}} \leq \sqrt{N} \|A\|$.
Furthermore, if the entries of $A$ are polynomials of degree at most $d$, then
$(\|A\|_{\mathrm{HS}})^2$ is a polynomial of degree at most $2d$.

\begin{lemma}
 Let $A(\ul{x})$ be a $n\times n$ matrix whose entries are trigonometric
 polynomials of degree $d$. Then the condition
 \be
  \|A(\ul{x})^{-1}\|_{\mathrm{HS}} \leq C
 \ee
 defines a semi-algebraic set of degree $\leq 2nd $
 and the condition
 \be
  |\spr{\delta_k}{A(\ul{x})^{-1} \delta_{\ell}}| \leq C
 \ee
 a semi-algebraic set of degree $\leq 2 nd $.
\end{lemma}

\begin{proof}
 $\det(A(\ul{x}))$ is a trigonometric polynomial
 of degree at most $n d$. As the first condition is equivalent
 to
 \[
  \sum_{\text{minors $\ti{A}$ of $A$}}
   \det(\ti{A})^2 \leq C \det(A)^2
 \]
 the claim follows. The second claim is similar.
\end{proof}

\begin{proof}[Proof of Theorem~\ref{thm:semialgdesc}]
 By Corollary~\ref{cor:approxzLM}, there exists $A(\ul{x})$
 such that
 \[
  \|(z (\mathcal{L}^{[a,b]}_{\ul{x};\beta,\ti{\beta}})^* - 
   \mathcal{M}^{[a,b]}_{\ul{x};\beta,\ti{\beta}}) - A(x)\| \leq \E^{-N^2}.
 \]
 Define the set $U$ as the set of all $\ul{x}$ such that
 \[
  \|A(\ul{x})^{-1}\| \leq \frac{1}{2^{p+1}} \E^{\Gamma}
 \]
 and for $k,\ell$ with $\gamma_{k,\ell} \neq 0$, we have
 \[
  |\spr{\delta_k}{A(\ul{x})^{-1}\delta_{\ell}}| \leq \frac{1}{2^{p}}\E^{-\gamma_{k,\ell}}.
 \]
 By the previous lemma, this is a semi-algebraic set of
 the right degree. 

 By Lemma~\ref{lem:perturbsuit}, we have the right inclusions
 with the sets $\mathcal{U}_{p-1}$ and $\mathcal{U}_{p+1}$.
\end{proof}

%
%
%

\section{Proof of Theorem~\ref{thm:mainA}}
\label{sec:proof}

The goal of this section is to prove Theorem~\ref{thm:mainA}.
Before beginning, with the proof, we will need to discuss some
additional facts about CMV operators. The basics can be found
in Section~\ref{sec:cmvop}.

\begin{lemma}
 Let $z_0 \in \partial\D$, $\beta,\ti{\beta}\in\partial\D$,
 and $\dist(z_0, \sigma(\mathcal{E}_{\beta,\ti{\beta}}^{[a,b]})) \geq \delta$.
 Then for $k,\ell \in [a,b]$
 \be 
  |G^{[a,b]}_{\beta,\ti{\beta}}(z_0; k,\ell)| \leq \frac{1}{\delta}.
 \ee
\end{lemma}

\begin{proof}
 By assumption $\|(\mathcal{E}_{\beta,\ti{\beta}}^{[a,b]} - z_0)^{-1}\| \leq \frac{1}{\delta}$
 and $z (\mathcal{L}^{[a,b]}_{\beta,\ti{\beta}})^{*}
   - \mathcal{M}^{[a,b]}_{\beta,\ti{\beta}} = (\mathcal{L}^{[a,b]}_{\beta,\ti{\beta}})^{*} \cdot
  (z - \mathcal{E}_{\beta,\ti{\beta}}^{[a,b]})$. 
 By $\|(\mathcal{L}^{[a,b]}_{\beta,\ti{\beta}})^{*}\| = 1$, the claim follows.
\end{proof}

We will need the following lemma relating the Green's function
on a large scale $[a,b]$ to the Green's function on a small
scale $[c,d]\subseteq [a,b]$.

\begin{lemma}\label{lem:gemresolv}
 Let $[c,d]\subseteq [a,b]$ be finite intervals in $\Z$,
 $\beta,\ti{\beta},\gamma,\ti{\gamma} \in \partial\D$, $z\in\partial\D$,
 and $k \in [c,d]$, $\ell\in [a,b] \setminus [c,d]$.
 Then 
 \be
  |G_{\beta,\ti{\beta}}^{[a,b]}(z; k, \ell)| \lesssim
   \sup_{n \in \{c,c-1,c+1, d,d-1,d+1\}} |G_{\beta,\ti{\beta}}^{[a,b]}(z; k, n)| \cdot
    \sup_{n \in \{c,c+1, d,d-1\}} |G_{\gamma,\ti{\gamma}}^{[c,d]}(z; n, \ell)|
 \ee
\end{lemma}

\begin{proof}
 Let $A=z (\mathcal{L}^{[a,b]}_{\beta,\ti\beta})^{*}- \mathcal{M}^{[a,b]}_{\beta,\ti\beta}$,
 $B =z (\mathcal{L}^{[c,d]}_{\gamma,\ti\gamma})^{*}- \mathcal{M}^{[c,d]}_{\gamma,\ti\gamma}$.
 Denote by $A_1$ the restriction of $A$ to $\ell^2([a,b]\setminus [c,d])$ and $\Gamma = A_1 \oplus B - A$.
 Note $A, B$ are unitary, but $A_1, \Gamma$ are not.

 By the resolvent equation, we have that
 \[
  A^{-1} - (A_1 \oplus B)^{-1} = A^{-1} \Gamma (A_1 \oplus B)^{-1}.
 \]
 As clearly $\spr{\delta_k}{(A_1 \oplus B)^{-1}\delta_{\ell}} = 0$,
 we obtain the claim by computing $\Gamma$ and collecting terms.
\end{proof}

Here and in the following, we use the notation
$A \lesssim B$ for that there exists a universal
constant $C > 0$ such that $A \leq C B$.
Of course, the previous lemma also holds for perturbing the boundary
conditions $\beta,\ti{\beta}$. We omit this in the statement to avoid
awkward notation.

\bigskip

We now return to the proof of Theorem~\ref{thm:mainA} by
recalling some results from \cite{kcmv1} and reformulating
them in a way that will be useful to us.

\begin{lemma}\label{lem:wegner2resolv}
 Let $\beta,\ti{\beta},z\in\partial\D$, $N\geq 1$, and $\eps > 0$.

 Then there exists a set $\mathcal{B}^{IDS}_{N}\subseteq\T^r$ of measure
 \be
  |\mathcal{B}^{IDS}_{N}| \leq \eps
 \ee
 such that for all $\ul{x}\in\T^r \setminus \mathcal{B}^{IDS}_{N}$
 and $k,\ell \in [-N,N]$
 \be
  \|(\mathcal{E}^{[k,\ell]}_{\ul{x};\beta,\ti{\beta}} - z)^{-1}\|
   \lesssim \frac{N^{3}}{\eps}.
 \ee
\end{lemma}

\begin{proof}
 By the Wegner estimate, Theorem~5.3. in \cite{kcmv1},
 we have for each choice of $k,\ell\in [-N,N]$
 that
 \[
  |\{\ul{x}:\quad \|(\mathcal{E}^{[k,\ell]}_{\ul{x};\beta,\ti{\beta}} - x)^{-1}\|
   > B\}| \lesssim \frac{N}{B}.
 \]
 Choosing $B$ such that the right hand side is $\frac{\eps}{ 2N^2}$
 and taking the union over all possible choices of $k,\ell$,
 the claim follows (there are $N (2N -1)$ many choices of $k$ and $\ell$).
\end{proof}

Next, we have the following result for the Green's function

\begin{proposition}\label{prop:greenesti1}
 Let $z\in\partial\D$ and $\eta > 1$, $N\geq 1$ large enough
 depending on $r$. Then there exists
 a set $\mathcal{B}_N^1 \subseteq \T^r$ such that
 \be
  |\mathcal{B}_N^{1}| \leq \frac{1}{N^{\frac{\eta}{2}}}
 \ee
 and for $\ul{x}\in\T^r\setminus\mathcal{B}_N^{1}$ we have that
 there exist $\beta, \ti{\beta} \in \T$ such that
 \be
  \|(\mathcal{E}^{[-N,N]}_{\ul{x};\beta,\ti{\beta}} - z)^{-1}\| \leq N^{\eta}
 \ee
 and for $|\ell|\leq\frac{N}{2}$ we have
 \be
  |G^{[-N,N]}_{\ul{x};\beta,\ti{\beta}}(z;\ell, \pm N)| \leq \frac{1}{N^{\eta}}.
 \ee
\end{proposition}

\begin{proof}
 Use \cite[Theorem~5.1.]{kcmv1} to obtain the necessary estimates
 on the resolvent and \cite[Theorem~7.1.]{kcmv1} to obtain the
 decay on the Green's function. The results of \cite{kcmv1} are only stated
 in the case $r = 2$, but using the results from Appendix~\ref{sec:returntimes},
 it is clear that they extend to $r \geq 3$.
\end{proof}

The next lemma draws a conclusion from Proposition~\ref{prop:greenesti1},
which allows us to use methods similar to multiscale analysis to prove
our main claim. The main idea is to trade probability for better
decay of the Green's function in space.

\begin{lemma}\label{lem:multiinit1}
 Let $\eta > 1$. There exists $c\in (0,1)$ such that for $N \geq 1$
 large enough there exists
 a set $\mathcal{B}\subseteq\T^{r}$ of measure
 \be
  |\mathcal{B}| \leq \frac{1}{N^{\frac{\eta}{4}}}
 \ee
 such that $\ul{x} \in \T^{r}\setminus\mathcal{B}$, we have
 for $k \in \{-N,N\}$ and $|\ell| \leq \frac{N}{2}$
 \be
  |G_{\ul{x}; \beta,\ti{\beta}}^{[-N,N]}(z;k,\ell)| \leq \E^{-\gamma |k-\ell|}
 \ee
 where $\gamma = \frac{1}{N^c}$.
\end{lemma}

\begin{proof}
 We denote by $\mathcal{B}^{1}_{M}$ the bad set from
 Proposition~\ref{prop:greenesti1}. Define
 \[
  \mathcal{B}^{1}_{M;L} = \bigcup_{\ell = -L}^{L} T^{\ell \lfloor\frac{M}{2}\rfloor} \mathcal{B}_{M}^{1}.
 \]
 As the skew-shift is measure preserving, we have
 that $|\mathcal{B}^{1}_{M;L}| \leq (2 L + 1) |\mathcal{B}^{1}_{M}|$.
 So taking $L = \lfloor \frac{1}{3} M^{\frac{\eta}{4}} \rfloor$,
 we obtain a set $\mathcal{B}^{2}_{L} = \mathcal{B}^{1}_{M;L}$
 with measure
 \[
  |\mathcal{B}^{2}_{L}| \leq \frac{1}{L^{\frac{\eta}{4}}}.
 \]
 Furthermore, an iteration of Lemma~\ref{lem:gemresolv}
 as discussed in Proposition~6.1. of \cite{kcorr}, we obtain
 that for $\ul{x} \in \mathcal{B}^{2}_{L}$, we have
 \[
  |G_{\ul{x};\beta,\ti{\beta}}^{[-L \lfloor\frac{M}{2}\rfloor,L\lfloor\frac{M}{2}\rfloor]}(z; k,\ell)| \leq \frac{1}{2} 
    \E^{- \frac{\eta \log(M)}{2 M} |k - \ell|}
 \]
 as long as $|k - \ell| \geq \frac{1}{10} L\lfloor\frac{M}{2}\rfloor$.
 We note that this is still subexponential decay in the size
 of the interval.

 We define $L$ and $M$ by $N = L \lfloor\frac{M}{2}\rfloor$,
 or equivalently $M = \lfloor \frac{1}{6} N^{\frac{4}{4+\eta}} \rfloor$
 and then $L$ as above. Thus, we see that the claim holds for
 \[
  c \in (\frac{4}{4 + \eta}, 1)
 \]
 and $N$ large enough.
\end{proof}

We define
\be
 \mathcal{U}_{\gamma,\tau,p}(N) =
  \{\ul{x}:\quad [-N,N]\text{ is not $(\gamma,\tau,p)$-suitable for } \mathcal{E}_{\ul{x}} - z\}.
\ee
With this new notation, we can formulate

\begin{corollary}\label{cor:multiinit2}
 Let $\eta > 1$, $\tau\in (0,1)$, and $p\geq 1$. There exists $c \in (0,1)$ 
 such that for $N \geq 1$ large enough, we have for $\gamma = \frac{1}{N^c}$ that
 \be
  |\mathcal{U}_{\gamma,\tau,p}(N)| \leq \frac{1}{N^{\eta}}.
 \ee
\end{corollary}

\begin{proof}
 This is an immediate consequence of the previous lemma
 and Lemma~\ref{lem:wegner2resolv}.
\end{proof}

We now come to the inductive result

\begin{proposition}\label{prop:multi2}
 There exist constants $\eta > 0$, $\sigma > 0$, $\tau \in (0,1)$, etc
 
 Let $L \geq 1$ be large enough. Then
 \be
  |\mathcal{U}_{\gamma,\ti\tau,3}(L)| \leq \frac{1}{L^{\eta}}
 \ee
 for some $\ti\tau\in(0,1)$
 implies for $N = L^{\frac{\eta}{2}}$ that
 \be
  |\mathcal{U}_{\gamma - N^{\tau-1},\tau,3}(N)| \leq \E^{-N^{\sigma}}.
 \ee
\end{proposition}

Our first goal is to pass from the probabilistic assumption
in the proposition to a statement in space.

\begin{lemma}\label{lem:probtospace}
 There exists $p=p(r) \in (0,1)$, $\eta > 1$ such that for every $\ul{x}\in\T^r$,
 $M \geq 1$ large enough, and $N \geq L^{\frac{1}{p^3}}$, we have that
 \be
  \#\{1 \leq n \leq N:\quad T^{n} \ul{x} \in \mathcal{U}_{\gamma,\tau,1}(L)\}
   \leq N^{p}
 \ee
 if
 \be
  |\mathcal{U}_{\gamma,\tau,2}(L)| \leq \frac{1}{N^{\eta}}.
 \ee
\end{lemma}

\begin{proof}
 By Theorem~\ref{thm:semialgdesc}, we have that there is a semi-algebraic $S$ set of 
 degree $\leq L^{C}$ and measure $|S| \leq \frac{1}{L^{\ti{\eta}}}$
 such that $\mathcal{U}_{\gamma,\tau,3}(L) \subseteq S$.
 The claim now follows by an application of Theorem~\ref{thm:returnsemialg}.
\end{proof}

From Lemma~\ref{lem:wegner2resolv}, we now obtain a set 
$\mathcal{B}^{\mathrm{IDS}}_{N}$ of measure
\be
 \E^{-N^{\sigma}}
\ee
such that for every $\ul{x}\in\T^{r}\setminus\mathcal{B}^{\mathrm{IDS}}_{N}$
and $-N\leq k < \ell \leq N$, we have
\be
 \|(\mathcal{E}_{\ul{x};\beta,\ti{\beta}}^{[k,\ell]} - x)^{-1}\|
  \leq \E^{L^{\tau}}.
\ee
Clearly, this implies the first estimate needed for the
proof of Theorem~\ref{thm:mainA}.

\begin{proof}[Proof of Proposition~\ref{prop:multi2}]
 Let $|\ell|\leq\frac{N}{2}$. We will just prove the estimate
 on $G^{[-N,N]}_{\ul{x};\beta,\ti{\beta}}(z;\ell, N)$. By Lemma~\ref{lem:probtospace},
 there exist a minimal $n_{1,+}$ such that $n_{1, +} > \ell$ and 
 \[
  |G_{\ul{x};\beta,\ti{\beta}}^{n_{1, +}+[-L,L]}(z; n_{1,+},
   n_{1, +} \pm L)| \leq \frac{1}{2} 
   \E^{- \gamma L}.
 \]
 $n_{1,-}$ is chosen
 in a similar fashion such that $n_{1, -} \leq \ell$.
 By Lemma~\ref{lem:gemresolv} and the estimate on the resolvent,
 we can thus conclude that
 \begin{align*}
  |G^{[-N,N]}_{\ul{x};\beta,\ti{\beta}}(z;\ell, N)| \leq \frac{\E^{-(\gamma - L^{\tau-1}) L}}{2}
   \Big(&|G^{[-N,N]}_{\ul{x};\beta,\ti{\beta}}(z;n_{1,+}+L, N)|\\
   &+ |G^{[-N,N]}_{\ul{x};\beta,\ti{\beta}}(z;n_{1,+}-L, N)|\\
   &+ |G^{[-N,N]}_{\ul{x};\beta,\ti{\beta}}(z;n_{1,-}+L, N)| \\
   &+ |G^{[-N,N]}_{\ul{x};\beta,\ti{\beta}}(z;n_{1,-}-L, N)| \Big).
 \end{align*}
 We can now choose $n_{2, \pm}$ similarly by requiring $|n_{2,\pm}| \geq
 |n_{1,\pm}| +L$. Iterating, we find
 \[
  |G^{[-N,N]}_{\ul{x};\beta,\ti{\beta}}(z;\ell, N)| \leq 2 \E^{L^{\tau}}
   \E^{- S L (\gamma-L^{\tau-1})}
 \]
 for $S$ the maximal time we can perform this operation. 

 It remains to estimate $S$. Let $\ell > 0$, then it clearly suffices
 to estimate the number of times, we can choose $n_{j, +}$. As there are
 $N - \ell - N^{p}$ many good points to choose from, and we always need
 to skip $L$ of these, we see that, we have
 \[
  S \geq \frac{N - \ell - N^{p}}{L}.
 \]
 This implies $S L \geq |N - \ell| \cdot (1 - \frac{2}{N^{1-p}})$,
 which clearly implies the claim.
\end{proof}

\begin{proof}[Proof of Theorem~\ref{thm:mainA}]
 By Corollary~\ref{cor:multiinit2}, we can apply Proposition~\ref{prop:multi2}
 inductively to conclude the claim at larger and larger scales. This yields the claim.
\end{proof}

\appendix

%
%
%

\section{Return times estimates}
\label{sec:returntimes}

We know that the skew-shift $T:\T^r\to\T^r$ is uniquely ergodic. In particular,
if $U\subseteq\T^r$ is an open set, we have that for every $\ul{x}\in\T^r$ 
\be
 \frac{1}{N} \#\{1 \leq n \leq N:\quad T^n \ul{x}\} \to |U|.
\ee
The main goal of this section will be to quantify this convergence for 
special sets called {\em semi-algebraic}, see Section~\ref{sec:semialg}.
The methods discussed here
can be essentially be extracted from \cite{bbook}.

\begin{theorem}\label{thm:returnsemialg}
 Let $\mathcal{S}$ be a semi-algebraic set of degree $B$
 and $\omega$ Diophantine.
 Then there is $q \in (0,1)$, $G \geq 1$, and $C \geq 1$ such that
 for $\ul{x}\in\T^r$
 \be
  \#\{1\leq \ell \leq L :\quad T^{\ell}\ul{x}\in\mathcal{S}\}
   \leq C B^G L^{1 - q}
 \ee
 as long as
 \be
  |\mathcal{S}| \leq \frac{1}{L^{\frac{qr}{2}}}.
 \ee
\end{theorem}

We will need the following estimate on return times

\begin{theorem}\label{thm:returnball}
 Let $\omega$ be Diophantine and $T_{\omega}:\T^r\to\T^r$ be the skew-shift
 with frequency $\omega$. Then
 \be
  \#\{1\leq \ell \leq L:\quad T^{\ell}_{\omega}\ul{x}\in B_{\eps}(\ul{a})\}
   \lesssim \eps^r L + \frac{1}{\eps^r} L^{1-p}
 \ee
 for an universal constant $p = p(r) \in (0,1)$.
\end{theorem}

In order to pass from balls to semi-algebraic sets,
we will need the following result

\begin{theorem}\label{thm:gromov}
 Let $\mathcal{S}\subseteq\T^r$ be a semi-algebraic set
 of degree $\deg(\mathcal{S}) \leq B$ for $B$ large enough
 and measure $|\mathcal{S}|\leq \eps^{r}$.
 Then for an universal constant $G \geq 1$, there exists
 $1 \leq T \leq B^G \eps^{-(r-1)}$ and points $x_1, \dots, x_T$ such that
 \be\label{eq:Scoveredepsballs}
  \mathcal{S} \subseteq \bigcup_{t=1}^{T} B_{\eps} (x_t).
 \ee
\end{theorem}

\eqref{eq:Scoveredepsballs} states that $\mathcal{S}$ can be
covered by less than $\eps^{-(r-1)} B^G$ many $\eps$ balls. Here, we denote
\be
 B_{\eps}(a) = \{x\in\T^K\quad \|x-a\|\leq \eps\}.
\ee

\begin{proof}[Proof of Theorem~\ref{thm:returnsemialg}]
 Combining these two theorems,
 we obtain
 \[
  \#\{1 \leq \ell \leq L:\quad T^{\ell}\ul{x} \in \mathcal{S}\}
   \lesssim B^G (\eps L + \frac{1}{\eps^{2r-1}} L^{1 - p}).
 \]
 Hence, taking $\eps L = \frac{1}{\eps^{2r -1}} L^{1-p}$ or equivalently
 \[
  \eps = \frac{1}{L^{\frac{p}{2r}}}
 \]
 we obtain that
 \[
  \#\{1 \leq \ell \leq L:\quad T^{\ell}\ul{x} \in \mathcal{S}\}
   \lesssim B^G L^{1 - \frac{p}{2r}}
 \]
 which is the claim.
\end{proof}

%
%
%

\end{document}